\documentclass[12pt]{amsart}
\usepackage{a4ams}
\usepackage{amsfonts}
\usepackage{amssymb}
\usepackage[arrow,matrix,curve]{xy}
\usepackage{ifthen}

\usepackage[active]{srcltx}

\newtheorem{theorem}{Theorem}[section]
\newtheorem{definition}[theorem]{Definition}

\newtheorem{lemma}[theorem]{Lemma}
\newtheorem{proposition}[theorem]{Proposition}
\newtheorem{corollary}[theorem]{Corollary}

\newcommand{\N}{{\mathbb N}}

\newcommand{\Z}{{\mathbb Z}}
\newcommand{\F}{{\mathbb F}}

\newcommand{\C}{{\mathbb C}}

\newcommand{\T}{{\mathbb T}}

\renewcommand{\P}{{\mathbb P}}

\def\cal{\mathcal}

\newcommand{\cala}{{\cal A}}
\newcommand{\calo}{{\cal O}}

\newcommand{\calb}{{\cal B}}
\newcommand{\calp}{{\cal P}}
\newcommand{\calh}{{\cal H}}

\newcommand{\calt}{{\cal T}}

\newcommand{\cals}{{\cal S}}

\newcommand{\calq}{{\cal Q}}
\newcommand{\calx}{{\cal X}}

\newcommand{\alg}[0]{ {\rm Alg} }

\newcommand{\ring}{ {\rm Ring} }

\newcommand{\aut}[0]{ \mbox{\rm{Aut}}}

\newcommand{\set}[2]{\{ \, #1 \,\, | \, \, #2 \, \} }

\newcommand{\lin}[0]{ \mbox{\rm{span}}}

\begin{document}

\title{On freely generated semigraph $C^*$-algebras}
\author[B. Burgstaller]{Bernhard Burgstaller}
\address{Doppler Institute for mathematical physics, Trojanova 13, 12000 Prague, Czech Republic}
\email{bernhardburgstaller@yahoo.de}
%

\begin{abstract}
For special
universal $C^*$-algebras associated to $k$-semigraphs
we present the universal representations of these algebras, prove a Cuntz--Krieger uniqueness theorem, and compute the $K$-theory.
These $C^*$-algebras seem to be the most universal Cuntz--Krieger like algebras naturally associated to $k$-semigraphs.
For instance, the Toeplitz Cuntz algebra is a proper quotient of such an algebra.
\end{abstract}

\maketitle


\section{Introduction}

In this paper we continue our study of higher rank semigraph $C^*$-algebras from \cite{burgiSemigraphI}.
The class of higher rank semigraph $C^*$-algebras present a flexible generalisation of Tomforde's ultragraph algebras \cite{tomforde}
or $C^*$-algebras of labelled graphs \cite{batespask} to higher rank
structures.

%

Let $T$ be a $k$-semigraph.
We define a special semigraph $C^*$-algebra $\calq(T)$ called a freely generated semigraph $C^*$-algebra.
It is generated by $T$ as a subset of partial isometries of $\calq(T)$, thereby satisfying only a minimum
set of relations such that we can still speak of a generalised Cuntz--Krieger algebra.
Notably, a Cuntz--Krieger uniqueness theorem holds.
We find a concrete faithful representation of $\calq(T)$ on a Hilbert space as a left regular representation of a semimultiplicative
set in Proposition \ref{PropRepresVarphi}.
A freely generated semigraph algebra is free, weakly free, cancelling, and satisfies a Cuntz--Krieger theorem, see Theorem \ref{TheoremSemigraphCStarToeplitz}.
We compute the $K$-theory of $\calq(T)$
in Theorem \ref{TheoremKTheory}.
The proof is just an application of
a theorem in \cite{burgiKTheoryGaugeActions}.



Every semigraph $C^*$-algebra is generated by partial isometries with commuting source and range projections (see Lemma \ref{semigraphwordrep}).
%
There is some recent interest in universal $C^*$-algebras generated by partial isometries and their $K$-theory,
see for instance Cho and Jorgensen \cite{chojorgensen} and Brenken and Niu \cite{brenkenniu}.
The strongest motivation for considering a freely generated semigraph algebra is that it seems to be somehow the freest Cuntz--Krieger
like algebra which naturally contains a given $k$-graph, see also Lemma \ref{lemmaExampleQuell} which further justifies this.
Moreover, we think every universal algebra generated by partial isometries for which one knows the $K$-theory is a benefit;
even more so, as there is (despite \cite{1181.46047}) an ongoing program initiated by G. Elliott to classify certain
subclasses of nuclear $C^*$-algebras (like the subclass of purely infinite nuclear $C^*$-algebras) where $K$-theory plays a major role.


We give a brief overview of this paper.
In Section \ref{sectionSemigraphAlgebra} we recall the theory of semigraph algebras.
In Section \ref{SectionFreeSemiAlgebra} we consider two conditions for a semigraph algebra
which we call weakly free (Definition \ref{defKToeplitz}) and free (Definition \ref{defToeplitz}).
Freeness implies weakly freeness (Proposition \ref{PropToeplitzKToeplitz}),
and free semigraph algebras are cancelling (Proposition \ref{PropToeplitzAmenability}), and so satisfy a Cuntz--Krieger uniqueness
theorem.
In Section \ref{SectionQuellSemigraph} we introduce freely generated semigraph algebras (Definition \ref{ToeplitzgrahDef}) and show
that they are free and cancelling (Theorem \ref{TheoremSemigraphCStarToeplitz}).
In Section \ref{SectionKTheoryWeaklyFreeSemigraph} we explicitly compute the $K$-theory of weakly free semigraph
$C^*$-algebras in Theorem \ref{TheoremKTheory}.
In Section \ref{SectionKTheoryQuellSemigraph} we use Theorem \ref{TheoremKTheory} for the computation of the $K$-theory of a freely generated semigraph $C^*$-algebra in Theorem \ref{theoremKTheoryQuell}.
In Section \ref{SectionConcludingRemarks} we give some final remarks, including examples and a discussion which known results are generalised by the paper.

\section{Semigraph algebras}  \label{sectionSemigraphAlgebra}


In this section we recall briefly the definition and some basic facts about semigraph algebras \cite{burgiSemigraphI}
for further reference.
%


%

\begin{definition}  \label{definitionSemimultiSet}
{\rm
A {\em semimultiplicative set} $T$ is a set equipped with
a subset $T^{\{2\}} \subseteq T \times T$ and a multiplication
$T^{\{2\}} \longrightarrow T : (s,t) \mapsto st$,
which is associative, that is, for all $s,t,u \in T$,
$(s t) u$ is defined if and only if $s (t u)$ is defined, and
both expressions are equal if they are defined.
}
\end{definition}

\begin{definition}  \label{Defsemigraph}
{\rm
Let $k$ be an index set (which may be regarded as a natural number if $k$ is finite).
A {\em $k$-semigraph} $T$ is a semimultiplicative set $T$ equipped with a degree map
$d:T \longrightarrow \N_0^k$ satisfying the {\em unique
factorisation property} which consists of the following two conditions:

(1) For all $x,y \in T$ for which the product $xy$ is defined one has
$d(xy)= d(x)+d(y)$.

(2) For all $x \in T$ and all $n_1,n_2 \in \N_0^k$ with $d(x)
= n_1+n_2$ there exist unique $x_1,x_2 \in T$ with $x = x_1 x_2$ satisfying $d(x_1)=n_1$ and $d(x_2)=n_2$.
%
}
\end{definition}

We call a $k$-semigraph also a higher rank semigraph or just a semigraph.
We shall also write $|t|$ rather than $d(t)$ for elements $t$ in a $k$-semigraph.
Let $T$ be a $k$-semigraph. We denote the set of all elements of $T$ with degree $n$
by $T^{(n)}$ ($n \in \N_0^k$).
If $x \in T$ and $0 \le n_1
\le n_2 \le d(x)$ then there are unique $x_1,x_2,x_3 \in T$ such that $x = x_1
x_2 x_3$, $d(x_1)=n_1, d(x_2)= n_2-n_1$ and $d(x_3) = d(x)-n_2$.
The element $x_2$ is denoted by $x(n_1,n_2)$.

\begin{definition}
{\rm
Let $T$ be a $k$-semigraph.
For $s,t \in T$ we write $s \le t$ if $\alpha s=t$ for some $\alpha \in T$.
%
}
\end{definition}

\begin{lemma}   \label{lemmaOrderrel}
The last relation is an order relation on a semigraph.
\end{lemma}

\begin{proof}
The relation is reflexive since $d(x)=0+d(x) \Rightarrow x=sx$ for some $s \in T$.
Transitivity is clear. So
assume $s \le t$ and $t \le s$. Then there are $\alpha,\beta \in T$ such that
$t = \alpha s$ and $s = \beta t$. Then $s = \beta \alpha s$, and so $d(s)= d(s) + d(\alpha)+ d(\beta)$,
which implies $d(\alpha)=d(\beta)=0$.
Hence one has $\beta \alpha s = \beta \alpha \beta \alpha s$. By the unique factorisation property
$s = \alpha s$. So $s =t$.
\end{proof}


\begin{definition}
{\rm A $k$-semigraph $T$ is called {\em finitely aligned} if for all
$x,y \in T$ the
{\em minimal common extension} of $x$
and $y$, which is the set
\begin{eqnarray*}
T^{(\min)}(x,y) &=& \{(\alpha,\beta) \in T \times T |\,
\mbox{$x \alpha$ and $y \beta$ are defined},\\
&& x \alpha = y \beta,\, d(x \alpha) =
d(x) \vee d(y) \},
\end{eqnarray*}
is finite. }
\end{definition}
It might help to recall the meaning of the relation $(\alpha,\beta) \in T^{\rm (min)}(x,y)$ by visualising it by the tautologies
$\alpha \le x \alpha$ and $\beta \le y \beta$.

\begin{definition}   \label{defNonunitalSemigraph}
{\rm
$T$ is called a {\em non-unital} $k$-semigraph if there exists a $k$-semigraph $T_1$ which has a unit $1 \in T_1$ such that
$T= T_1 \backslash \{1\}$.
}
\end{definition}

We shall use the following notions when we speak about algebras.
A {\em $*$-algebra} means an algebra
over $\C$ endowed with an involution. An element $s$ in a
$*$-algebra is called a {\em partial isometry} if $s s^* s= s$,
and a {\em projection} $p$ is an element with $p = p^2 = p^*$.
We define $P_a := a a^*$ and $Q_a := a^* a$ for elements $a$ of a $*$-algebra.
If $s$ is a partial isometry then $P_s$ and $Q_s$ are called the {\em range} and {\em source projection} of $s$, respectively.
If $I$ is a subset of a $*$-algebra then $\langle I \rangle$ denotes the self-adjoint two-sided
ideal generated by $I$ in this $*$-algebra.

Assume that we are given a set $\calp$ and a non-unital $k$-semigraph $\calt$ with $\calp \cap \calt = \emptyset$.
Recall that in a non-unital $k$-semigraph one has $d(t)>0$ for all its elements $t$ (because $t 1 = 1 t$, and thus
the unique factorization property implies $t=1$ if $d(t)=0$, see the last paragraph in \cite[Section 3]{burgiSemigraphI}).
We denote by $\calt_1 = \calt \sqcup \{1\}$ the unital $k$-semigraph of Definition \ref{defNonunitalSemigraph}.
Define $\F$ to be the free non-unital $*$-algebra generated by $\calp \cup \calt$.
We call a $*$-monomial $a_1^{\epsilon_1} \ldots a_n^{\epsilon_n}$ with letters $a_i$ in $\calp \cup \calt$ and exponents
$\epsilon_i \in \{1,*\}$ a word (in $\F$ or in a quotient of $\F$).

\begin{definition} \label{degreeMap}
{\rm
The {\em degree}
$d(x)$ of a {word} $x=x_1 \ldots x_n$ in $\F$ ($n \ge 1$, $x_i
\in \calp \cup \calt \cup \calp^* \cup \calt^*$) is defined to be
$d(x)=d(x_1) + \ldots + d(x_n)$,
where $d(x_i)$ is to be the
semigraph-degree $d(x_i)$ when $x_i \in \calt$, $d(x_i) = 0$ if $x_i \in
\calp$, and $d(x_i^*) = - d(x_i)$ for any $x_i \in \calt \cup \calp$.
}
\end{definition}

\begin{definition}   \label{defFiberSPace}
{\rm
The {\em fiber space} of $\F$ is the union of all fibers $\lin(W_n)$, where $\lin$ means linear span and $W_n$ denotes the set of words with degree $n \in \Z^k$.
}
\end{definition}

\begin{definition}   \label{definitionGaugaActions}
{\rm
Let $\sigma: \T^k \longrightarrow \aut(\F)$ be the {\em gauge action}
defined by
$\sigma_\lambda (p) = p$ and
$\sigma_\lambda(t) = \lambda^{d(t)} t$
for all $p \in \calp,t \in \calt$ and $\lambda \in \T^k$.
}
\end{definition}

\begin{lemma}[\cite{burgiSemigraphI}, Lemma 4.7]  \label{lemmaGaugeActions}
Let $X$ be a $*$-algebra which is a quotient of $\F$. Then $X$ is actually the quotient of $\F$ by a subset of the fiber space if and only if
there is a gauge action on $X$ which is canonically induced by the corresponding gauge action on $\F$ (see Definition \ref{definitionGaugaActions}).
\end{lemma}

\begin{definition}[Semigraph algebra]    \label{DefSemigraphAlgebra}
{\rm A {\em $k$-semigraph algebra} $X$ is a $*$-algebra which is
generated by disjoint subsets $\calp$ and $\calt$ of $X$, where

\begin{itemize}
\item[(i)]
$\calp$ is a set of commuting projections closed under taking
multiplications,

\item[(ii)]
$\calt$ is a set of nonzero partial isometries closed under nonzero
products,

\item[(iii)]
%
$\calt$ is a non-unital finitely aligned $k$-semigraph,

\item[(iv)]
for all $x \in \calt$ and all $p \in \calp$ there is a $q \in \calp$
such that $p x = x q$,

\item[(v)]
for all $x,y \in \calt$ there exist $q_{x,y,\alpha,\beta} \in \calp$
such that
\begin{equation}  \label{defExpansion}
x^* y = \sum_{(\alpha,\beta) \in \calt_1^{(\min)}(x,y)} \alpha q_{x,y,\alpha,\beta} \beta^*, \mbox{ and}
\end{equation}

\item[(vi)]
$X$ is canonically isomorphic to the quotient of $\F$ by a subset of the fiber space
(Definition \ref{defFiberSPace}).
\end{itemize}
}
\end{definition}

It is understood in identity (\ref{defExpansion}) that the unit $1$ in $\calt_1 = \calt \sqcup \{1\}$ is also a unit for all elements of $X$.
The universal $C^*$-algebra $C^*(X)$ generated by $X$ is called the $k$-semigraph $C^*$-algebra associated to $X$.
We often write also $\sum_{x \alpha = y \beta}$ rather than $\sum_{(\alpha,\beta) \in \calt_1^{(\min)}(x,y)}$
in (\ref{defExpansion}).
\begin{definition}
{\rm
%
We call an element of $\{ s p t^* \in X|\, s,t \in \calt_1, p \in\calp\}$ a {\em standard word}
(of the semigraph algebra $X$).
We call an element of $\{ s p \in X|\,s \in \calt_1, p \in\calp\}$ a {\em half-standard word}.
}
\end{definition}


For a subset $Z$ of a $*$-algebra $X$ we set
\begin{eqnarray*}
\P(Z) &=& \set{P_{x} (1-P_{y_1}) \ldots (1-P_{y_m}) \in X}{ x,y_i \in Z,\,  m \ge 0}.
\end{eqnarray*}

\begin{definition}
{\rm

We call an element of
$\P(\{sp| s \in \calt_1, p \in \calp\})$
a {\em standard projection}
(of the semigraph algebra $X$).
}
\end{definition}

\begin{lemma}[\cite{burgiSemigraphI}, Lemma 5.8] \label{semigraphwordrep}
\begin{itemize}
\item[(a)]
The word set of $X$ is an inverse semigroup of partial isometries.

\item[(b)]
For each word $w$ there exist half-standard words $a_i,b_i$ and $c_i$ such
that
$w w^* = \sum_{i=1}^n a_i a_i^*$ and
$w = \sum_{j =1}^m b_j
c_j^*$
with $d(w)=d(b_s c_s^*)$ for all $1 \le s \le m$.
\end{itemize}
\end{lemma}

\begin{corollary}   \label{corollarySpannedStandardWords}  \label{corollaryInvSemigroup} \label{lemmaSourceProjHSword}
(a) A semigraph algebra is linearly spanned by its standard words.

(b) The range projection of a word is a sum of range projections of half-standard words.

(c) The source projection of a half-standard word is in $\calp$.
\end{corollary}

The {\em core} of a semigraph algebra is the linear span of all its words with degree zero.
It forms a $*$-subalgebra of the semigraph algebra.

\begin{corollary}[\cite{burgiSemigraphI}, Corollary 6.4] \label{corollarySemigraphMatrixunits}
The core is the union of a net of finite dimensional $C^*$-algebras,
each one allowing a matrix representation where each projection on
the diagonal is a finite sum of mutually orthogonal standard projections.
A $C^*$-representation of $X$ is injective on
the core if and only if it is non-vanishing on nonzero standard projections.
\end{corollary}

\begin{definition}   \label{defCancel}
{\rm A semigraph algebra $X$ is called {\em cancelling} if for every
standard word $w$ with nonzero degree and every nonzero
standard projection $p$ there is a nonzero standard projection
$q$ such that $q \le p$ and $q w q = 0$.
}
\end{definition}

\begin{theorem}[Cuntz--Krieger uniqueness theorem, \cite{burgiSemigraphI}, Theorem 7.3] \label{theoremamenablesemigraphsystem}

Let $X$ be a cancelling semigraph algebra.
Then
the universal representation $X \longrightarrow C^*(X)$ is injective on
the core, and so non-vanishing on the nonzero standard projections,
and up to
isomorphism this is the only existing representation of $X$ in a $C^*$-algebra which
is non-vanishing on nonzero standard projections and has dense image.
\end{theorem}


\section{Free semigraph algebras}
\label{SectionFreeSemiAlgebra}

In this section we consider a condition on a semigraph algebra called freeness, and a weaker one called weakly freeness.
The rough idea of freeness is that range projections of generators should not sum up to a unit.
Freeness is thus a generalisation of the defining inequality  $P_{s_1} + P_{s_2} < 1$ in the Toeplitz version of the Cuntz algebra $\calo_2$
with generating isometries $s_1$ and $s_2$ to arbitrary higher rank semigraph algebras.
%
%
In this section $X$ denotes a semigraph algebra.

\begin{definition}
{\rm
Write $X^{\backslash i}$
for the $*$-subalgebra of $X$ generated by the elements $\calp$
and the non-unital finitely aligned semigraph
$\calt^{\backslash i}= \set{t \in \calt}{d(t)_i =0}$
(the $i$th coordinate vanishes). We set $\calt_1^{\backslash i} = \calt^{\backslash i} \sqcup \{1\}$.
}
\end{definition}

\begin{lemma}   \label{lemmaXbacki}
$X^{\backslash i}$ is a semigraph algebra.
\end{lemma}

\begin{proof}
All points (i)-(v) of Definition \ref{DefSemigraphAlgebra} are
almost obvious. Point (vi) can be seen by Lemma
\ref{lemmaGaugeActions} and the fact that the restriction of the gauge action on $X$ to $X^{\backslash i}$ is a gauge action on $X^{\backslash i}$.
\end{proof}

We shall regard $X^{\backslash i}$ as a sub-semigraph algebra of $X$.
When we work in $X$ and say ``$w$ is a word in $X^{\backslash i}$'' (and so on)
then ``word" refers to the semigraph algebra $X^{\backslash i}$; so $w$ does not mean a word in $X$
which accidentally happens to be in $X^{\backslash i}$.

 %
%

%

\begin{definition}   \label{defKToeplitz}
{\rm A semigraph algebra $X$ is called {\em weakly free} if for all
coordinates $i \in k$, all nonzero standard projections $p$ of $X^{\backslash i}$
and all finite subsets
$\calb \subseteq \calt^{(e_i)}$ the inequality
$p \le \sum_{b \in \calb} P_b$
does not hold. }
\end{definition}

The next two lemmas (Lemma \ref{lemmamini22} and Lemma \ref{lemmaKToeplitz}) will only be used in the proof of Lemma \ref{proofcondab}.
The reader only interested in Theorem \ref{TheoremKTheory} could go directly to Section \ref{SectionKTheoryWeaklyFreeSemigraph} after these two lemmas.

\begin{lemma}   \label{lemmamini22}
If $i \in k, a \in \calt^{(e_i)}$ and $w$ is a word in $X^{\backslash i}$
then there exist $b_1, \ldots, b_n \in \calt^{(e_i)}$ such that
$P_a w = P_a w (P_{b_1} + \ldots
+ P_{b_n})$.
\end{lemma}

\begin{proof}
Since every word can be written as a sum of standard words, we may assume that $w$ is a standard word and write $w = s p t^*$ for $s,t \in \calt_1^{\backslash i}$ and $p \in \calp$.
Then (by Definition \ref{DefSemigraphAlgebra} (v))
\begin{eqnarray*}
(a a^*) (s p t^*)
&=& a \sum_{a \alpha =  s \beta} \alpha q_{a,s,\alpha,\beta} \beta^* p t^* \\
&=&  \sum_{a \alpha =  s \beta, \,t \beta \neq 0} a \alpha q_{a,s,\alpha,\beta} p_\beta \beta^*  t^* \\
&=&  \sum_{a \alpha =  s \beta , \, t \beta \neq 0} a \alpha q_{a,s,\alpha,\beta} p_\beta t'^* \beta'^*
\sum_{a \alpha =  s \beta ,\, t \beta \neq 0} P_{\beta'}\\
\end{eqnarray*}
for $p_\beta \in \calp$ with $p \beta = \beta p_\beta$ (by Definition \ref{DefSemigraphAlgebra} (iv)), and where $d(\beta) = d(a)= e_i$,
and $t \beta = \beta' t'$ by the unique factorisation property with $d(\beta')= e_i$
(if $t \beta \neq 0$, which, recall, is equivalent to $t \beta$ being a defined product in $\calt$).
\end{proof}

\begin{lemma} \label{lemmaKToeplitz}
If $X$ is
weakly free then for every coordinate $i \in k$, every element $x$ of the fiber space of $X^{\backslash i}$,
and every
finite subset $\calb$ of  $\calt^{(e_i)}$ we have
that $(\sum_{b \in \calb} P_b ) x = x$ implies $x=0$.
\end{lemma}

\begin{proof}
Suppose that $x$ is a nonzero element of the $0$-fiber $X_0^{\backslash i}$ (i.e. the core) of $X^{\backslash i}$, $p= \sum_{b
\in \calb} P_b$, and $p x = x$. The element $x$
is in the core of $X^{\backslash i}$ and so in a finite dimensional
$C^*$-algebra $A$ as described in Corollary
\ref{corollarySemigraphMatrixunits}. Say, $x= \sum_{ij} \lambda_{ij}
e_{ij}$ for matrix units $e_{ij}$, where each diagonal unit $e_{ii}$
is a sum of mutually orthogonal standard projections of $X^{\backslash i}$ (Corollary
\ref{corollarySemigraphMatrixunits}). Since $x \neq 0$ we may
suppose that $0 \neq \lambda_{i_0 j_0} = 1$ for some fixed pair
$(i_0,j_0)$. By Corollary \ref{corollarySemigraphMatrixunits}, there
is a nonzero standard projection $q$ in $X^{\backslash i}$ (so $q \in \P(\calt_1^{\backslash i} \calp)$) such that $q \le e_{i_0 i_0}$. Note that $p$ commutes with $e_{i_0 i_0}$ since standard projections commute.
Then
$$q \le e_{i_0 i_0} = e_{i_0 i_0} x e_{j_0 i_0} =   e_{i_0 i_0} p x
e_{j_0 i_0} = p e_{i_0 i_0} x e_{j_0 i_0} = p e_{i_0 i_0} \le p.$$
However, $q \le p$ contradicts the weakly free condition in $X$.
Thus $x = 0$.

Now assume that $x$ is in another fiber $X_n^{\backslash i}$ ($n \in \Z^k$) and $p x = x$
(where $p= \sum_{b\in \calb} P_b$ again).
We may write $x$ as
$x = \sum_{k=1}^l a_k x_k b_k^*$,
where $x_k$ is an element of the core $X_0^{\backslash i}$, $a_k \in \calt_1$ with $d(a_k) = \max(n,0)$,
$b_k \in \calt_1$ with $d(b_k^*)= \min(n,0)$, and the pairs $(a_k,b_k)$ are mutually distinct for different $k$'s, for all $1 \le k \le l$. Fix $1 \le k_0 \le l$. Then, since $px=x$, and by several applications of Lemma \ref{lemmamini22},
there exists a subset $\calb' \subseteq \calt^{(e_i)}$ such that for $p'=\sum_{b \in \calb'} P_b$ one has
$$r_{k_0}:= a_{k_0}^* x b_{k_0} = Q_{a_{k_0}} x_{k_0} Q_{b_{k_0}} = a_{k_0}^* p x b_{k_0} = p' a_{k_0}^* p x b_{k_0}
= p' r_{k_0}.$$
Thus $r_{k_0} = p' r_{k_0}$ and $r_{k_0} \in X_0^{\backslash i}$, and so by what we have already
proved, $r_{k_0} = 0$. Since $k_0$ was arbitrary,
$$0= \sum_{k_0=1}^l a_{k_0} r_{k_0} b_{k_0}^* = \sum_{k_0=1}^l a_{k_0} x_{k_0} b_{k_0}^* = x.$$
\end{proof}

\begin{definition}   \label{defToeplitz}
{\rm A semigraph algebra $X$ is called {\em free} if
\begin{itemize}
\item[(i)]
$P_a$ is not in $\calp$ for every half-standard word $a \not \in \calp$, and

\item[(ii)] if $p \in \calp$ and $a_1, \ldots, a_n$ are half-standard words with $P_{a_i} < p$ for all $1 \le i \le n$ then
$\bigvee_{i=1}^n P_{a_i}
< p$.


\end{itemize}
}
\end{definition}

\begin{lemma} \label{lemmaEquivToeplitzCond}
Let $X$ be free.
If
$a,b_1, \ldots , b_n$ are half-standard words
satisfying $P_a  (1-P_{b_i}) \neq 0$ for all $1 \le i \le n$
then
$P_a \prod_{i=1}^{n} (1- P_{b_i} ) \neq 0$.
\end{lemma}

\begin{proof}
We may write $a=x s$ and $b_i=y_i p_i$ for certain
$x,y_i \in \calt_1$ and $s,p_i \in \calp$. Assume that $X$ is
free and $P_a (1-P_{b_i}) \neq 0$ for all $1 \le i
\le n$. 
We have (by Definition \ref{DefSemigraphAlgebra} (v))
\begin{eqnarray*}  \label{somesmallerformula}
P_a (1-P_{b_i}) &=& x s x^* (1-y_i p_i y_i^*)\\
&=& x s x^* - \sum_{x \alpha = y_i \beta} x s \alpha q_{x,y_i,\alpha,\beta} \beta^* p_i y_i^*\\
&=& x s x^* \Big(1 - \sum_{x \alpha = y_i \beta} x s \alpha q_{x,y_i,\alpha,\beta} p_{i,\beta} \alpha^* x^* \Big)\\
&=& x s x^* \prod_{(\alpha,\beta) \in
\calt_1^{(\min)}(x,y_i)} (1 -  x s \alpha q_{x,y_i,\alpha,\beta} p_{i,\beta}
\alpha^* x^*),
\end{eqnarray*}
where $p_{i,\beta}$ is chosen in $\calp$ such that $p_i \beta = \beta p_{i,\beta}$ (by Definition \ref{DefSemigraphAlgebra} (iv)),
and where in the last identity we successively used the formula $(1-p)(1-q)= 1 - p - q$ for orthogonal projections $p$ and $q$.
Since the above is nonzero by assumption, we have
$$x s
\alpha q_{x,y_i,\alpha,\beta} p_{i,\beta} \alpha^* x^* < x s x^*$$
for all $1 \le
i \le n$ and all $(\alpha,\beta) \in \calt_1^{(\min)}(x,y_i)$.
Multiplying here from the left and right with $x^*$ and $x$, respectively, we get
$x^* x s \alpha q_{x,y_i,\alpha,\beta} p_{i,\beta} \alpha^* < x^* x s$.
By freeness (Definition \ref{defToeplitz}) we conclude
$$\bigvee_{i=1}^n \bigvee_{x\alpha = y_i \beta} x^* x s \alpha q_{x,y_i,\alpha,\beta} p_{i,\beta} \alpha^* < x^* x s.$$
Thus $\bigvee_{i,\alpha,\beta} x s \alpha q_{x,y_i,\alpha,\beta} p_{i,\beta} \alpha^* x^* < x s x^*$, whence
\begin{eqnarray}  \label{expandStandardProj}
P_a \prod_{i=1}^{n} (1- P_{b_i} )  &=&  x s x^* \prod_{i=1}^n \prod_{x\alpha = y_i \beta} (1- x s \alpha q_{x,y_i,\alpha,\beta} p_{i,\beta} \alpha^* x^*)\\
&=&
x s x^* \Big (1 - \bigvee_{i,\alpha,\beta} x s \alpha q_{x,y_i,\alpha,\beta} p_{i,\beta} \alpha^* x^* \Big) \neq 0,
\end{eqnarray}
where we used de Morgan's law $\bigwedge (1-\gamma) = (1-\bigvee \gamma)$.
\end{proof}


Similarly as for elements in a semigraph we write
$x_1 \le x_2$ for half-standard words $x_1$ and $x_2$ if they allow a representation
$x = t_1 p_1$ and $x_2 = t_2 p_2$ ($t_1,t_2 \in \calt$ and $p_1,p_2 \in \calp$)
with $t_1 \le t_2$.
Only in the next corollary $\calh$ denotes the set of half-standard words.

\begin{corollary} \label{semigraphstructureP1T1P}
If $X$ is free then these two sets are the set of nonzero standard projections:
\begin{eqnarray*}
&& \Big \{P_a \prod_{i=1}^n (1-P_{b_i}) \in X \Big |\,n \in \N_0,\, a, b_i
\in \calh,
\, P_a (1-P_{b_i}) \neq 0 \mbox{ for all $i$} \Big \}\\
&=& \Big \{P_a \prod_{i=1}^n (1-P_{b_i}) \in X  \Big |\,  n \in \N_0,\, a, b_i
\in \calh,
\, P_{b_i} < P_{a} \mbox{ and } a \le b_i \,\forall i\Big \}.
\end{eqnarray*}
%
%
\end{corollary}

\begin{proof}
That the first set is the set of nonzero standard projections follows from Lemma
\ref{lemmaEquivToeplitzCond}. For the second set just recall the identity (\ref{expandStandardProj})
how we can write down a standard projection.
The assertion can be directly read off from this expansion.
\end{proof}

\begin{proposition} \label{PropToeplitzKToeplitz}
If $X$ is free then
$X$ is weakly free.
\end{proposition}

\begin{proof}
Fix $i \in k$.
To
check weakly freeness, consider a finite subset $\calb$
of $\calt^{(e_i)}$ and a nonzero standard projection $p = x q x^* \prod_{i=k}^n
(1-y_k y_k^*)$ in $X^{\backslash i}$, where
$x \in \calt_1^{\backslash i}, q \in \calp$ and the $y_k$'s are half-standard words in
$X^{\backslash i}$.
If $b \in
\calb$ then
\begin{equation}  \label{estimate1}
q x^* x \ge q x^* b b^* x  =
\sum_{(e,f)\in \calt_1^{(\min)}(x,b)} q (e q_{x,b,e,f} f^*) (f
q_{x,b,e,f} e^*)
\end{equation}
\begin{equation}  \label{estimate2}
=
\sum_{(e,f)\in \calt_1^{(\min)}(x,b)} P_{e q_e q_{x,b,e,f} f^*  f}  < q x^* x ,
\end{equation}
where we used $q e = e q_e$ from Definition \ref{DefSemigraphAlgebra} (iv).
The last inequality is here by freeness. Indeed,
$P_{e q_e q_{x,b,e,f} f^*  f} \notin \calp$ by Definition \ref{defToeplitz} (i).
On the other hand, $P_{e q_e q_{x,b,e,f} f^*  f} \le  q x^* x \in \calp$ by (\ref{estimate1}),
so
$P_{e q_e q_{x,b,e,f} f^*  f} <  q x^* x$.
Hence, inequality (\ref{estimate2}) is true by Definition \ref{defToeplitz} (ii).

We conclude from (\ref{estimate1}) and (\ref{estimate2})
that $q x^* P_b x \neq x^* x q$.
Hence, applying here the operation $x (\cdot ) x^*$ we get $x q x^* P_b
\ne x q x^*$.
Thus, $x q x^* (1-P_b) \neq 0$ for all $b \in \calb$.
Hence
$$p \Big(1- \sum_{b \in \calb}P_b \Big ) = p \prod_{b \in \calb} (1-P_b)
= x q x^* \prod_{i=k}^n
(1-y_k y_k^*) \prod_{b \in \calb} (1-P_b)  \neq 0$$
by Lemma
\ref{lemmaEquivToeplitzCond}. Consequently
$p \le \sum_{b \in
\calb} P_b$
does not hold.
\end{proof}

\begin{proposition} \label{PropToeplitzAmenability}
If $X$ is free then $X$ is cancelling.
\end{proposition}

\begin{proof}
We are going to check the cancelling condition, Definition
\ref{defCancel}. Let $w = \alpha q \beta^*$ be a standard word
with $d(w) \neq 0$ ($\alpha,\beta \in \calt_1, q \in
\calp$). Let $P$ be a nonzero standard projection.
We must find a nonzero standard projection $Q$ with $Q \le P$ and $Q x Q=0$.

We may write $P=p x s x^*$,
where $p= \prod_i (1-y_i y_i^*)$, $x \in \calt_1, s \in \calp$
and the
$y_i$'s are half-standard words.
If already $P w P=0$ then the cancelling condition
is verified.
So assume that
$$0 \neq P w P = p x s (x^* \alpha) q \beta^* p x s
x^*.$$
Then there is a pair $(e,f) \in \calt_1^{(\min)}(x,\alpha)$
such that
$$v:=p x s (e q_{x,\alpha, e,f} f^*) q \beta^* p x s
x^* \neq 0$$
by (\ref{defExpansion}).
Consequently $p x s e \neq 0$, and so
\begin{equation}  \label{formulaPprime}
0 \ne P' := p x s e
e^* x^* = p x e s' e^* x^*  = p x' s' x'^*
\end{equation}
is a standard projection,
where $x' := x e$ and $s' \in \calp$ satisfies $s e = e s'$. (We intensively use the fact
that the word set forms an inverse semigroup, Lemma
\ref{semigraphwordrep}.) Note that $P' \le P$. If already $P' w P' = 0$,
then the cancelling condition is verified.

So assume $P' w P' \neq 0$.
Note that $|x'| = |x e| \ge |\alpha|$ since $(e,f) \in
\calt_1^{(\min)}(x,\alpha)$.
Note that by (\ref{formulaPprime}) $P' = p x' s' x'^*$ has the same shape as $P$,
but with $|x'| \ge |\alpha|$.
As we are going to search $Q \le P' \le P$, we may assume without loss of generality
that we are given $P$ with $P w P \neq 0$ and $|x| \ge |\alpha|$.
Similar computations as above on the $\beta$-side show that we may also assume, by choosing a smaller projection
than $P$, that also $|x| \ge |\beta|$.


So assume without loss of generality that $P w P \neq 0$ and $|x| \ge |\alpha|, |\beta|$. We
have
$$0 \neq P w P = p x s  x^*  (\alpha q \beta^*) x s  x^* p.$$
Hence, there are decompositions $x= \alpha x_1 = \beta x_2$
($x_1,x_2 \in \calt_1$), and a $q' \in \calp$ chosen to satisfy $Q_\alpha q Q_\beta x_2 =  x_2 q'$,
such that
$$PwP =p x s (\alpha x_1)^* \alpha q \beta^* \beta x_2  s x^* p
= p x s (x_1^*  x_2) q' s x^* p$$
$$= p x s \sum_{(e,f) \in \calt_1^{(\min)}(x_1,x_2)} e q_{x_1,x_2,e,f}
f^* q' s  x^* p$$
$$= p x s K s x^* p,$$
where $K:= \sum_{(e,f) \in \calt_1^{(\min)}(x_1,x_2)} e q_{x_1,x_2,e,f} f^* q'$.
Set
$$
Q' := \prod_{(e,f) \in
\calt_1^{(\min)}(x_1,x_2)} (1- e e^*).$$
Successively using the formula $x s (1-e e^*) (1- e' e'^*) s x^* = x s (1- e e^*) s x^* x s (1-e' e'^*) s x^*$,
we get
\begin{equation}
Q := p x s Q' s x^*  = p  \prod_{(e,f) \in \calt_1^{(\min)}(x_1,x_2)} (1- x s e e^* x^*) x s x^*
\label{formelQ}
\end{equation}
\begin{equation}   \label{formelQ2}
= x s x^* \prod_i (1-y_i y_i^*) \prod_{(e,f) \in \calt_1^{(\min)}(x_1,x_2)} (1- x s e e^* x^*).
\end{equation}
In (\ref{formelQ2}) we have substituted $p= \prod_i (1-y_i y_i^*)$. By (\ref{formelQ2}), $Q$ is a standard projection.
Recall that $P = p x s x^*$, so $Q \le P$.
Since $Q' K = 0$, we obtain
$$Q w Q= Q (P w P) Q = Q p x s K s x^* p Q = Q p x s K s x^* p p x s Q' s x^* = 0.$$
So $Q$ is the sought standard projection that cancels $w$. It remains to show $Q \neq 0$.

Since $d(\alpha q
\beta^*) \neq 0$, we may assume without loss of generality that
$|x_1| < |x_1| \vee |x_2|$, that is, $|e| \neq 0$ for every $e$ in (\ref{formelQ}). By freeness, Definition \ref{defToeplitz} (i), $(x^* x) s e e^* = P_{x^* x s e}  \notin
\calp$, and so $x^* x s e e^* \neq x^* x s$. Hence $x s e e^*
x^* \neq x s x^*$. This shows $0 \neq (1 - x s e e^* x^* ) x s x^*$.
Also $0 \neq (1- y_i y_i^*) x s x^*$ since $0 \neq P = p x s x^*$.
Hence, by formula (\ref{formelQ2}) and Lemma \ref{lemmaEquivToeplitzCond} $Q \neq 0$.
\end{proof}



\begin{lemma}   \label{lemmaImageSemigraphUniversal}  \label{corollaryFreeCancelCstar}
Let $X$ be a semigraph algebra and $\pi: X \longrightarrow C^*(X)$ the universal representation.
Then $\pi(X)$ is a semigraph algebra (which is free if $X$ is free).
\end{lemma}

\begin{proof}
By universality of the universal representation $\pi:X \longrightarrow C^*(X)$, the gauge map
$X \stackrel{\sigma_\lambda}{\longrightarrow} X \stackrel{\pi}{\longrightarrow}  C^*(X)$
($\lambda \in \T^k$), Definition
\ref{definitionGaugaActions}, induces a gauge map $\tilde \sigma_\lambda: C^*(X) \longrightarrow C^*(X)$.
%
Hence, by Lemma \ref{lemmaGaugeActions}, $\pi(X)$ is a quotient of the free algebra $\F$ by a subset of the fiber space.
Thus $\pi(X)$ must be a quotient of $X$ by a subsets of its fiber space, and is thus
a semigraph algebra by \cite[Lemma 8.1]{burgiSemigraphI}.
Now assume that $X$ is free.
By the uniqueness theorem (Theorem \ref{theoremamenablesemigraphsystem}) the cores of $X$ and $\pi(X)$ are isomorphic,
and so the validity of Definition \ref{defToeplitz} (ii) carries over from $X$ to $\pi(X)$.
If a half-standard word $\pi(a)$ in $\pi(X)$ ($a$ being a half-standard word in $X$) is not in $\pi(\calp)$
then $d(a)=d(\pi(a))>0$, and so $P_a \notin \calp$ since $X$ is free, and thus $P_{\pi(a)} = \pi(P_a)\notin \pi(\calp)$.
This verifies Definition \ref{defToeplitz} (i) for $\pi(X)$.
\end{proof}

%


\section{Freely generated semigraph $C^*$-algebras}

\label{SectionQuellSemigraph}


In this section we define freely generated semigraph algebras. Let us anticipate roughly what it is.
A freely generated semigraph algebra could be most simply explained by considering a
higher rank graph $T$ \cite{kumjianpask}. Then the freely generated semigraph $C^*$-algebra $\calq(T)$ for $T$ is a $C^*$-algebra
which is similar to the Toeplitz graph algebra $\calt C^*(T)$ \cite{raeburnsims} but without the relations
$Q_{t} = s(t)$ ($t \in T$). So the Toeplitz graph algebra is a quotient of the freely generated semigraph algebra.


Suppose that $T$ is a finitely aligned $k$-semigraph.
Note that if $e \in T^{(0)}$ and $e^2$ is defined then $e$ is
automatically idempotent. Indeed, by the factorisation property we
may choose a unique decomposition $e = ab$ for certain $a,b \in T^{(0)}$.
Then $e^2 = abab$.
By the unique factorisation property $e= a = b$, and so $e^2 = a b = e$,
which proves the claim.
Moreover, if
$e \in T^{(0)}$ is idempotent and $x \in T$ then either $ex$ is
undefined or $ex =x$ (since $e x = e^2 x$ and so $x = ex$ by the
unique factorisation property).
In particular, $e f$ must be undefined
for distinct $e,f \in T^{(0)}$.

{In this section it is assumed that a semigraph $T$ has only
idempotent elements in $T^{(0)}$.}
Let us summarise the consequences in a lemma.

\begin{lemma}   \label{lemmaidempotentsT}
The elements of $T^{(0)}$ are idempotent elements and mutually incomposable.
If $e \in T^{(0)}, t \in T$ and $e t$ is defined then $e t =t$.
(Similarly, $t e= t$ if $t e$ is defined.)
\end{lemma}

\begin{definition}
{\rm
For $x \in T$
we define $r(x) = x(0,0)$ ({range} of $x$) and $s(x) = x(d(x),d(x))$ ({source} of $x$).
}
\end{definition}



\begin{definition}[Freely generated semigraph algebra for $T$]  \label{ToeplitzgrahDef}
{\rm Suppose that $T$ is a finitely aligned $k$-semigraph. Then one associates the {\em freely generated semigraph algebra}
$X$ to $T$. It is the universal $*$-algebra $X$
generated by the set $T$ subject to the following relations.
\begin{itemize}
\item[(i)]
$T$ consists of partial isometries,

\item[(ii)]
$T^{(0)}$ consists of projections,

\item[(iii)]
$X$ respects the multiplication of $T$ (that is, if $x y =z$ holds in
$T$ for $x,y,z \in T$ then this identity should also hold in $X$),

\item[(iv)]
$x y = 0$ for all $x,y \in T$ whose product $x y$ is undefined,

\item[(v)]
$Q_x$ and $Q_y$ commute for all $x,y \in T$, and

\item[(vi)]
\begin{equation} \label{xy_formula}
x^* y = \sum_{(e,f) \in T^{(\min)}(x,y)} e Q_{y f} f^*
\end{equation}
for all $x,y \in T$.
\end{itemize}
}
\end{definition}


We remark that $Q_{yf}$ means $f^* y^* y f$ in (\ref{xy_formula}).
Note that $T$ is faithfully embedded in the free algebra $\F$, but could perhaps degenerate
in the quotient $X$. Soon we will see below (Corollary \ref{corollaryQuellTfeithfullyEmbed}) however
that $T$ is also faithfully embedded in $X$.
That is why 
we will identify the $k$-semigraph $T$ with its embedding in $X$.
%
For the remainder of this section we shall assume that $T$ is a finitely aligned $k$-semigraph and
$X$ its associated free semigraph algebra.
%
%
For further remarks about the relations in Definition \ref{ToeplitzgrahDef},
see Section \ref{SectionConcludingRemarks}.



\begin{definition}
{\rm
We shall denote an element $(y, \alpha_1, \ldots, \alpha_n) \in
T^{n+1}$ ($n \ge 1$) symbolically by $y
\mu_{\alpha_1,\ldots,\alpha_n}$, or for brevity, by $y \mu_\alpha$.
Define
$$\Delta_\mu \,\,= \,\,\set{y \mu_{\alpha_1,\ldots,\alpha_n}}
{n \in \N, \,y,\alpha_i \in T \mbox{ and } \exists i \in \{1,\ldots, n\} \mbox{ with } y \le \alpha_i}.$$
}
\end{definition}

\begin{definition}
{\rm
Set $\Delta= T \sqcup \Delta_\mu$.
}
\end{definition}

\begin{lemma}
$\Delta$ is a semimultiplicative set.
\end{lemma}

\begin{proof}
We
endow $\Delta$ with the multiplication from $T$ for products within
$T$ (as far as defined), and define
\begin{equation} \label{identitymu}
x (y \mu_{\alpha}) = (xy)
\mu_{\alpha}
\end{equation}
if $x \in T, y \mu_\alpha \in \Delta_\mu$ and $(xy)
\mu_\alpha \in \Delta_\mu$. Other products in $\Delta$ are not
allowed (for example, products within $\Delta_\mu$ are undefined, or products where
an element of $\Delta_\mu$ appears as a left factor are invalid).

We are going to check that $\Delta$ is a semimultiplicative set. To this end we have to check associativity in the sense of Definition \ref{definitionSemimultiSet}.
Suppose that $a,b,c \in \Delta$. If $a,b,c \in T$ then $(ab) c$ is defined if and only if $a(bc)$ is defined as $T$ is a semimultiplicative set. If $a$ or $b$ is in $\Delta_\mu$ then
both $(a b) c$ and $a (b c)$ are undefined. Suppose $a,b \in T$ and $c \in \Delta_\mu$.
Write $c = y \mu_\alpha$. If $(a b) c = ((a b) y) \mu_\alpha$ is defined then $(ab)y \le \alpha_i$
for some $i$. Thus
$\alpha_i = z a b y$ for some $z \in T$ and so also
$a (b c)= a ((by) \mu_\alpha) = (a (by)) \mu_\alpha$
is defined and we have $(ab)c = a (b c)$. Similarly we see the reverse conclusion.
\end{proof}

If a semimultiplicative set $G$ has left cancellation, that means,
$s t_1 = s t_2$ implies $t_1 = t_2$ (for all $s,t_1,t_2 \in G$) then
we can associate a left reduced $C^*$-algebra to $G$ as defined next.
(We shall write $e_i$ or $\delta_i$ for the delta function $1_{\{i\}}$.)

\begin{definition}   \label{definitionReducedAlgebra}
{\rm
For a semimultiplicative set $G$ with left cancellation define
$\lambda:G \longrightarrow B(\ell^2(G))$ by
$$\lambda_s \Big (\sum_{t \in G} \alpha_t \delta_t \Big)= \sum_{\mbox{\small $t \in G, \, st $ is defined}} \alpha_t \delta_{s t}$$
for all $s \in G$ and $\alpha_t \in \C$.
The sub-$C^*$-algebra of $B(\ell^2(G))$ generated by $\lambda(G)$ is called the {\em left reduced
$C^*$-algebra} of $G$ and denoted by $C_r^*(G)$.
}
\end{definition}

\begin{proposition}   \label{PropRepresVarphi}
There is a representation
$\varphi: X \longrightarrow C^*_{r}(\Delta): \varphi(t)= \lambda_t$.
\end{proposition}


\begin{proof}
%
Of course, $\varphi: \F \longrightarrow C^*_r(\Delta)$ would be a well defined representation of the free algebra
$\F$ generated by $T$.
We need to show that this $\varphi$
respects the defining relations of Definition \ref{ToeplitzgrahDef}.
Note that the $\lambda_t$'s are partial isometries with commuting range and source projections
(these are canonical projections onto $\ell^2(Z)$ for subsets $Z \subseteq \Delta$).
So, by the property of $\Delta$ to be a semimultiplicative set and identity (\ref{identitymu}) the points (i)-(v)
of Definition \ref{ToeplitzgrahDef} (for $\varphi(T)= \lambda_T$ rather than $T$) are easy to see.
(For (ii) recall Lemma \ref{lemmaidempotentsT}.)

Let us write down the adjoint operators $\varphi(t)^*$. We have
\begin{equation} \label{identityVarphiTstar}   
\varphi(t)^* \delta_{t s \mu_\alpha}  =  \delta_{s \mu_\alpha}, \quad
\varphi(t)^* \delta_{t s}  =  \delta_s, \quad
\varphi(t)^* \delta_{a}  = 0 \,\, \mbox{(else)}
\end{equation}
To check Definition \ref{ToeplitzgrahDef} (vi), consider $x,y \in T$.
Suppose
\begin{equation}   \label{projectionVarphi}
\varphi(x x^* y y^*) \delta_a = \delta_a
\end{equation}
for an $a \in \Delta$. Then $a$ is a product $a= y a_y$ (since $\varphi(y)^* \delta_a \neq 0$) for some $a_y \in \Delta$,
and similarly $a = x a_x$ for some $a_x \in \Delta$.
Say that $a= y s_y \mu_\alpha = x s_x \mu_\alpha$.
Then $v:= y s_y$ has degree $d(v) \ge d(x) \vee d(y)$, and so there must exist a
minimal common
extension $(e,f) \in T^{\rm (min)}(x,y)$ such that
$$v (0, |x| \vee |y|) = x e = y f.$$
%
Consequently we have
\begin{equation} \label{compli1}
\varphi \Big(\sum_{(e,f) \in T^{\rm (min)}(x,y)} (x e)(y f)^* \Big) \delta_a = \delta_a.
\end{equation}
On the other hand, (\ref{projectionVarphi}) follows quite immediately from (\ref{compli1}).
Since, for arbitrary $a \in \Delta$, the right hand sides of (\ref{projectionVarphi}) and (\ref{compli1})
either give $\delta_a$ or zero, we conclude from the shown equivalence
that
\begin{equation}  \label{compli3}
\varphi(x x^* y y^*)  =  \varphi \Big(\sum_{(e,f) \in T^{\rm (min)}(x,y)} (x e)(y f)^* \Big).
\end{equation}
The range projection of $\varphi(e f^* y^*)$ is
$$\varphi(e (y f)^* (y f) e^*) = \varphi(e (x e)^*(x e) e^*) = \varphi(e e^* x^* x e e^*) = \varphi(x^* x e e^*),$$
which is a smaller projection than $\varphi(x^* x)$.
Thus,
$$\varphi(x^* x e f^* y^* y )= \varphi(e f^* y^* y) = \varphi(e f^* y^* y f f^*).$$
Hence, multiplying in (\ref{compli3}) from the left and right with $x^*$ and $y$, respectively,
we see that the identity (\ref{xy_formula}) holds in the image of $\varphi$.
\end{proof}

\begin{corollary}    \label{corollaryQuellTfeithfullyEmbed}
The canonical map $\iota:T \longrightarrow X$ is an injective $k$-semigraph homomorphism and non-degenerate.
\end{corollary}

\begin{proof}
We compose $\iota$ with $\varphi$ of Proposition \ref{PropRepresVarphi} to see this.
Let $t \in T$.
Let $e =s(t)$, so $t = t e$. Then $\varphi(\iota(t)) \delta_e = \delta_{t} \neq 0$.
So $\iota$ is non-degenerate.
If $s \in T$ is distinct from $t$ then we easily see with Lemma \ref{lemmaidempotentsT} that $\varphi(\iota(t))\delta_e = \delta_t \neq \varphi(\iota(s)) \delta_e$.
So $\iota$ is injective.
\end{proof}

\begin{lemma}    \label{lemmaQuellSemigraphAlgebra}
$X$ is a semigraph algebra with generators
\begin{eqnarray}
\calp  &=& \set{ Q_{t_1} \ldots Q_{t_n}\in X}{ n \in \N , \,t_i \in T} \cup \{0\}, \label{freesemigraphP}\\
\calt &=& T\backslash T^{(0)}
\end{eqnarray}
\end{lemma}

\begin{proof}
We need to show Definition \ref{DefSemigraphAlgebra}.
That $\calp$ is a commuting set of projections closed under multiplications (Definition \ref{DefSemigraphAlgebra} (i)) follows from
Definition \ref{ToeplitzgrahDef} (i) and (v).
That $\calt$ is a set of partial isometries closed under nonzero products (Definition \ref{DefSemigraphAlgebra} (ii)) follows from
Definition \ref{ToeplitzgrahDef} (i), (iii) and (iv).
We are going to check that $\calt_1$ is a semigraph (Definition \ref{DefSemigraphAlgebra} (iii)).
Let $t \in \calt_1$, and $t=t_1 t_2$ be the unique decomposition in $T$ subject to $m=d(t_1)>0$ and $n=d(t_2)>0$. This is the required decomposition in $\calt_1$ also.
If however $m=0$, say, then take the factorisation $t= 1 t$.

To prove (\ref{defExpansion}) of Definition \ref{DefSemigraphAlgebra} (v), consider $x,y \in \calt$.
Note that in (\ref{xy_formula}) $x e = y f$ and $Q_{y f} \in \calp$,
so (\ref{xy_formula}) looks already similar like (\ref{defExpansion}).
We only have to take care whether $\calt_1^{({\rm min})}$ and $T^{({\rm min})}$ make here a difference.
If $d(x) > d(y)$, say, then we have $\{(1,f)\}= \calt_1^{({\rm min})}(x,y)$
and $\{(e',f)\} = T^{({\rm min})}(x,y)$. Thus $x e' = y f$ so that $e' \in T^{(0)}$ is a
right unit for $x e' = y f$ by Lemma \ref{lemmaidempotentsT}, that is,
$y f e' = y f$. Hence, by the unique factorisation property in $T$, even $f e' = e'$.
Thus
$1 f^* y^* y f f^* = e' f^* y^* y f f^*$,
so there is no difference. The cases $d(x) < d(y)$ and $d(x)=d(y)$ are treated similarly.
If $d(x)$ and $d(y)$ are incomparable then there is obviously no difference.
To check Definition \ref{DefSemigraphAlgebra} (iv), just note that
$$Q_x y = (x^* x) y = x^* (x y) = y Q_{xy} 1 = y Q_{xy}$$
by (\ref{xy_formula}) if $x,y \in T$ and $xy \neq 0$.
The algebra $X$ is generated by $\calp$ and $\calt$ since
$e = e^* e = Q_e \in \calp$ for $e \in T^{(0)}$ by Definition \ref{ToeplitzgrahDef} (ii).
There is a gauge action $\sigma$ on $X$ given by
$\sigma_\lambda(t)= \lambda^{d(t)} t$ for $t \in T, \lambda \in \T^k$.
Indeed, it exists on the free algebra $\F$ generated by $T$, and so also on $X$, as the relations
of Definition \ref{DefSemigraphAlgebra} are invariant under the $\sigma_\lambda$'s.
Hence, Lemma \ref{lemmaGaugeActions} verifies Definition \ref{DefSemigraphAlgebra} (vi).
\end{proof}

\begin{definition}
{\rm
Let $T$ be a semigraph and $X$ its free semigraph algebra.
Then $\calq (T):= C^*(X)$ is called the {\em freely generated semigraph
$C^*$-algebra} associated to $T$.
}
\end{definition}

Let us use the following abbreviation. If $\alpha=(\alpha_1, \ldots, \alpha_n) \in T^n$
then $\mu_\alpha$ denotes $s(\alpha_1) \mu_\alpha$.
The next lemma shows us that the representation $\varphi$ can distinguish the elements of $\calp$.
The restriction $a \mapsto \varphi(a)|_{\ell^2(T)}$ is not able to do this, and this is why we considered
$\Delta$ at all and not just $T$, which would have been much simpler.

\begin{lemma}   \label{lemmaPalpha}
Let $p_\alpha = Q_{\alpha_1} \ldots Q_{\alpha_n}$ be nonzero for $\alpha=(\alpha_1, \ldots, \alpha_n)$ in $T^n$.
Then
$\varphi(p_\alpha) \delta_{\mu_\alpha} = \delta_{\mu_\alpha}$
and $\varphi(q) \delta_{\mu_\alpha} = 0$
for every $q \in \calp$ with $q < p_\alpha$.
\end{lemma}

\begin{proof}
Since $Q_{\alpha_i} = \alpha_i^* \alpha_i = \alpha_i^* \alpha_i s(\alpha_i)$,
and $p_\alpha$ is nonzero, $s(\alpha_1) = s(\alpha_i)$ for all $i$
by the orthogonality of the idempotent elements of $T$ (Lemma \ref{lemmaidempotentsT} and Definition \ref{ToeplitzgrahDef} (iv)).
Consequently,
\begin{equation}   \label{computeQmu}
\varphi(Q_{\alpha_i}) \mu_\alpha =
\varphi(\alpha_i^*) \varphi(\alpha_i) (s(\alpha_1) \mu_\alpha) =
\varphi(\alpha_i^*) (\alpha_i s(\alpha_1) \mu_\alpha)
= \mu_\alpha.
\end{equation}
This proves $\varphi(p_\alpha) \delta_{\mu_\alpha} = \delta_{\mu_\alpha}$.
We may write $q = Q_{y_1} \ldots Q_{y_l}$ for some
$y_i \in T$ (see (\ref{freesemigraphP})).
Note that either $\varphi(q) \mu_\alpha = \mu_\alpha$ or $\varphi(q) \mu_\alpha =0$.
Assume the first case.
Then by a similar computation as in (\ref{computeQmu}) we see that for every $i$ there is a $j_i$ such that $y_i \le \alpha_{i_j}$. Thus, for every $i$
we have $Q_{y_i} \ge Q_{\alpha_{i_j}} \ge p_\alpha$.
Hence $q \ge p_\alpha$, which is a contradiction to the assumption $q < p_\alpha$.
\end{proof}

\begin{theorem} \label{TheoremSemigraphCStarToeplitz}
The freely generated semigraph algebra $X$ associated to a $k$-semigraph $T$ is free, weakly free
and cancelling.
It thus satisfies the Cuntz--Krieger uniqueness theorem, that is, there is only one $C^*$-representation of $X$
(up to isomorphism) which is non-vanishing on nonzero standard projections.
This universal $C^*$-representation, which is also injective on the core, is
the representation $\varphi$ from Proposition \ref{PropRepresVarphi}.
In particular, there is an isomorphism between the freely generated semigraph $C^*$-algebra and the left reduced $C^*$-algebra of the semimultiplicative set $\Delta$, i.e.
$\calq(T) \cong   C^*_r(\Delta)$.
\end{theorem}

\begin{proof}
We are going to check that $X$ is free (Definition
\ref{defToeplitz}). Let $p \in \calp$ and $y_1,\ldots,y_n$ be half-standard words.
Assume that $P_{y_i} < p$ for all $1 \le i \le n$.
By (\ref{freesemigraphP}) and Lemma \ref{lemmaPalpha} there is an $\alpha \in T^n$ such that $p = p_\alpha$.
Let $\varphi: X \rightarrow C^*(\Delta)$ be the representation of Proposition \ref{PropRepresVarphi}.
By (\ref{identityVarphiTstar}) we have
$\varphi(y_i^*) \delta_{\mu_\alpha} = 0$ if $d(y_i) > 0$.
On the other hand,
if $d(y_i)= 0$ then $P_{y_i} \in \calp$, and so
$\varphi(y_i) \delta_{\mu_\alpha} =0$ by Lemma \ref{lemmaPalpha} (as $y_i = P_{y_i} < p_\alpha$).
Summarising these facts we get
$$\varphi \Big(\bigvee_{i=1}^n P_{y_i} \Big)\delta_{\mu_\alpha} = 0$$
and $\varphi(p_\alpha) \delta_{\mu_\alpha} = \delta_{\mu_\alpha}$ (Lemma \ref{lemmaPalpha}). Consequently
$\bigvee_{i=1}^n P_{y_i} < p_\alpha$. This proves
Definition \ref{defToeplitz} (ii).

If $y$ is a half-standard word with $d(y) >0$ then $\varphi(P_y) \delta_{\mu_\beta} =0$
for any $\beta$ by (\ref{identityVarphiTstar}). Consequently, $P_{y}$ cannot be in $\calp$
by Lemma \ref{lemmaPalpha}.
This verifies Definition \ref{defToeplitz} (i).

We are going to check that $\varphi$ is faithful on standard projections.
By Lemma \ref{semigraphstructureP1T1P} a nonzero standard projection $p$ allows a representation
$p=P_a \prod_{i=1}^n (1-P_{b_i})$ with $P_{b_i} < P_a$ and $a \le b_i$.
Say that $a = t p_\alpha = t Q_t p_\alpha$ for $t \in \calt_1$ and some $\alpha \in T^n$ according to Lemma \ref{lemmaPalpha}.
We may incorporate $Q_t$ in $p_\alpha$ and assume that $\alpha_1 = t$.
We have
\begin{equation}  \label{identtmua2}
\varphi(P_a) \delta_{t \mu_\alpha} = \delta_{t \mu_\alpha}
\end{equation}
by (\ref{identityVarphiTstar}) and Lemma \ref{lemmaPalpha}.
On the other hand, if $d(b_i) > d(a)$ then
\begin{equation}    \label{identtmua}
\varphi(P_{b_i}) \delta_{t \mu_\alpha} = 0
\end{equation}
by (\ref{identityVarphiTstar}). If $d(b_i) > d(a)$ is not true then $b_i = t p_\beta$ (for some $p_\beta \in \calp$)
since $a \le b_i$, and then, as $P_{b_i} < P_a$, i.e. $t p_\beta t^* < t p_\alpha t^*$,
one has
$$t^* t p_\beta t^* t < t^* t p_\alpha t^* t = p_\alpha$$
(the last identity by the fact that $\alpha_1 = t$).
Hence, $\varphi(p_\beta Q_t) \delta_{\mu_\alpha} = \varphi(p_\beta) \delta_{\mu_\alpha}= 0$ by Lemma \ref{lemmaPalpha}.
So also in this case we have (\ref{identtmua}).
Identities (\ref{identtmua2}) and (\ref{identtmua})
show that $\varphi(p) \neq 0$.
By Corollary \ref{corollarySemigraphMatrixunits}, $\varphi$ is injective on the core.

Hence also the universal representation of $X$ must be injective on the core.
Since we have also checked that $X$ is free, $X$ is weakly free and cancelling by
Propositions \ref{PropToeplitzKToeplitz} and \ref{PropToeplitzAmenability}.
Thus, by the Cuntz--Krieger uniqueness theorem (Theorem \ref{theoremamenablesemigraphsystem}),
$\varphi$ is the universal representation, which implies $\calq(T) \cong \overline{\varphi(X)} \subseteq C_r^*(\Delta)$.

$C_r^*(\Delta)$ is generated by the operators $(\lambda_t)_{t \in T}$, since
the operators $\lambda_t$ are zero for $t \in \Delta_\mu$ (the composition $t s$ is invalid in $\Delta$ for any
element $t \in \Delta_\mu$).
Consequently, the image of $\varphi$ is dense in $C_r^*(\Delta)$ and so
$\calq(T) \cong \overline{\varphi(X)} =
C^*((\lambda_t)_{t \in T})
= C^*_r(\Delta)$.
\end{proof}



The next lemma is intended to serve as an example for a particular freely generated semigraph algebra.
Let $\zeta_n$ be the graph induced by the skeleton consisting of one vertex $\nu$
and $n$ arrows $s_1, \ldots, s_n$ starting and ending in this single vertex $\nu$;
$n$ may be any cardinal number.

\begin{lemma}  \label{lemmaExampleQuell}
The freely generated semigraph $C^*$-algebra
$\calq(\zeta_n)$ is the universal unital $C^*$-algebra generated by the
free inverse semigroup (of partial isometries) of $n$ generators $t_1, \ldots , t_n$ with the additional relations
that the range projections of these generators are mutually orthogonal, i.e.
$P_{t_i} P_{t_j} = 0$
for all $i \neq j$.
\end{lemma}

\begin{proof}
Set $A= C^*(1, t_1,\ldots,t_n)$. In $A$, $1$ is a unit and the words in the letters $t_i$ form an inverse semigroup (where the inverse element should be the adjoint element in $A$, that is, inverse semigroup elements happen to be partial isometries); moreover $P_{t_i} P_{t_j} = 0$ for $i \neq j$. $A$ is universal with respect to these relations.
The freely generated semigraph $C^*$-algebra $\calq(\zeta_n)$ is a semigraph $C^*$-algebra (Lemma \ref{lemmaQuellSemigraphAlgebra}).
We define the homomorphism
$$\alpha:A \longrightarrow \calq(\zeta_n): \alpha(1) = \nu, \, \alpha(t_i) = s_i.$$
Since the generators $s_i$ form an inverse semigroup by Lemma \ref{semigraphwordrep},
$P_{s_i} P_{s_j} = 0$ by Definition \ref{ToeplitzgrahDef} (vi), and $\nu$ is a unit
in $\calq(\zeta_n)$ by Definition \ref{ToeplitzgrahDef} (iii) and the fact that $v s_i = s_i v = s_i$,
the map $\alpha$ is a well defined homomorphism.

We define the inverse homomorphism
$$\beta:\calq(\zeta_n) \longrightarrow A: \beta(\nu) = 1, \, \beta(s_{i_1} \ldots s_{i_m}) = t_{i_1} \ldots t_{i_m}.$$
$A$ is generated by $\beta(\{1,s_1,\ldots,s_n\})$, so $\beta$ is surjective. We need to show that the relations of Definition
\ref{ToeplitzgrahDef}, for $\beta(\zeta_n)$ rather than $\zeta_n$, hold in $A$.
Definition \ref{ToeplitzgrahDef} (iii) is satisfied in the image of $\beta$ as $\beta$ is multiplicative.
Definitions \ref{ToeplitzgrahDef} (i)-(ii) and (iv)-(v) are obviously also correct in the image of $\beta$.
Definition \ref{ToeplitzgrahDef} (vi) is
$$x^* y  =  \left \{
                                   \begin{array}{ll}
                                     \nu Q_x \nu = Q_x  & \mbox{ if $x = y$ } \\
                                      0 &   \mbox{ if $\zeta_n^{\rm (min)}(x,y)=\emptyset$} \\
                                      a Q_y \nu = a a^* x^* x a  &   \mbox{ if $xa=y$ for some $a \in \zeta_n$} \\
                                   \end{array}  \right .
$$
The first case is tautological, the third one reduces to a tautology in an inverse semigroup,
so holds in the image of $\beta$. The second case we demonstrate for $x = s_1 s_2$ and $y= s_1 s_3 s_4$, say.
One has $\beta(x)^* \beta(y)=\beta(x)^* \beta(x) \beta(x)^* \beta(y) \beta(y)^* \beta(y)=0$ since
\begin{eqnarray*}
\beta(x) \beta(x)^* \beta(y) \beta(y)^* &=& t_1 t_2 t_2^* (t_1^* t_1) t_3 t_4 t_4^* t_3^* t_1^*\\
&=&  t_1 t_2 t_2^*  t_3 t_4 t_4^* t_3^*  (t_1^* t_1) t_1^* = 0,
\end{eqnarray*}
where we have used inverse semigroup rules (commutativity of projections) and the fact that $P_{s_2} P_{s_3} = 0$.
Since all relations of Definition \ref{ToeplitzgrahDef} evidently hold in the image of $\beta$,
$\beta$ must be a well defined homomorphism.
This proves the lemma as $\alpha$ and $\beta$ are inverses to each other.
%
%
\end{proof}


\section{$K$-theory of weakly free semigraph $C^*$-algebras}

\label{SectionKTheoryWeaklyFreeSemigraph}

%

In this section we are going to compute the $K$-theory
of a weakly free semigraph $C^*$-algebra by an application of \cite[Theorem 2.2]{burgiKTheoryGaugeActions}.
Since the setting of \cite{burgiKTheoryGaugeActions} is somewhat lengthy, we do not recall
it here but directly apply it to semigraph algebras.

\begin{definition}  \label{defSourceCoverGenerators}
{\rm
For a semigraph algebra $X$ we say the {\em source projections cover the generators} if
for every $p \in \calp$ there is a $t\in \calt$ such that $p \le Q_t$.
}
\end{definition}

Let $X$ be a $k$-semigraph algebra. Assume that $k$ is finite,
the universal representation $X \longrightarrow C^*(X)$ is injective
(we shall regard $X$ as a subset of $C^*(X)$),
and that the source projections cover the generators
(Definition \ref{defSourceCoverGenerators}).
Define
$\cala_i=\set{t p}{t \in \calt^{(e_i)}, p \in \calp} \backslash \{0\}$
for $1 \le i \le k$.

\begin{lemma}    \label{lemmaPartitionA}
$X$ is generated by $\cala:=\cala_1 \cup \ldots \cup \cala_k$.
\end{lemma}

\begin{proof}
Let $t \in \calt$.
Then
$t = t_1 \ldots t_n = (t_1
Q_{t_1}) \ldots  (t_n Q_{t_n})$
for certain $t_i \in \calt$ with $d(t_i) = e_{i_j}$.
So $t$ is a product of elements of $\cala$ as $Q_s \in \calp$ for any $s \in \calt$.
Let $p \in \calp$. Since the source projections cover the generators by assumption,
there is a $t \in \calt$ such that
$p = p Q_t = p t^*t p
= (t p)^* (t p)$.
This is a product of elements of $\cala$ again as we may write
$t p$ as $tp= t_1 (t_2 p)$ where $t_1,t_2 \in \calt$, $d(t_2)=e_i$, so $t_2 p \in \cala_{i}$, and $t_1$ may be further expanded as above.
\end{proof}

Note that by the unique factorisation property in $\calt$ every standard word $w$ may be written as
\begin{equation}  \label{standardrepresSemi}
w = a_1 \ldots a_n p b_m^* \ldots b_1^*
\end{equation}
for suitable letters $a_i,b_j$ in $\bigcup_{i=1}^k \calt^{(e_i)}$ and some $p \in \calp$.

Define $\calx$ to be $C^*(X)$. By Lemma \ref{lemmaPartitionA} $\calx$ is generated by a finitely partitioned alphabet $\cala$.
We have a gauge action (as defined in Definition \ref{definitionGaugaActions}) with respect to this alphabet on $\calx$
and consequently a degree map determined by $d(a_i) = e_i$ for $a_i \in \cala_i$.
We define $\cals$ to be the set of half-standard words.
Their range projections commute by Corollary \ref{corollaryInvSemigroup}.
The core is locally matrical by Corollary \ref{corollarySemigraphMatrixunits}.
We resolve the core by finite dimensional $C^*$-algebras as they are described in Corollary \ref{corollarySemigraphMatrixunits}, and in particular choose $D$ to be the set of all their diagonal entries.
Since the diagonal elements of these matrices of the core are expressible as direct sums of standard projections
(Corollary \ref{corollarySemigraphMatrixunits}),
and every standard projection can be written as a sum of range projections
of half-standard words (Lemma \ref{semigraphwordrep}),
we have $D \subseteq P$, where we set
\begin{eqnarray*}
P &:=& \lin_\Z \set{ P_x \in \calx}{x \in \cals}, \\
Q &:=& \alg^* \set{Q_x\in \calx }{x \in \cals}.
\end{eqnarray*}
We define $W'$ to be the set of standard words. They linearly span $X$ by Lemma \ref{corollarySpannedStandardWords}.
We have all requirements for \cite[Theorem 2.2]{burgiKTheoryGaugeActions}, except the technical conditions (a) and (b)
from \cite{burgiKTheoryGaugeActions}.

\begin{lemma} \label{proofcondab}
If $X$ is weakly free
then the technical conditions (a) and (b) from \cite{burgiKTheoryGaugeActions}
hold.
\end{lemma}

\begin{proof}
The proof is very similar (an almost word by word translation) of \cite[Lemma 2.4]{burgiKTheoryGaugeActions}.
So we ask the reader to prove it along the lines of \cite[Lemma 2.4]{burgiKTheoryGaugeActions}
by using representation (\ref{standardrepresSemi})
and Lemma \ref{lemmamini22}, Lemma \ref{lemmaKToeplitz} and Lemma \ref{lemmaSourceProjHSword}
where necessary.
\end{proof}

We have now all requirements for \cite[Theorem 2.2]{burgiKTheoryGaugeActions} which states the following.

\begin{theorem}[\cite{burgiKTheoryGaugeActions}, Theorem 2.2] \label{ktheorytheorem}
The identical embedding $\theta:C^*(Q) \longrightarrow \calx$ induces an isomorphism
$K_0(\theta)$, and $K_1(\calx) =0$.
\end{theorem}

By (\ref{freesemigraphP}) we have $Q = \lin(\calp)$. Since the projections $\calp$ commute,
we have
$K_0(C^*(Q)) = \ring(\calp)$,
where $\ring(\calp)$ denotes the subring of $C^*(X)$ generated by $\calp$, regarded then as an abelian group under addition.
Since $Q$ is a subset of the core, which is locally matrical, $C^*(Q)$ is an AF-algebra
and thus $K_1(C^*(Q))=0$.
Theorem \ref{ktheorytheorem} states that the $K$-theory of $C^*(X)$ is the $K$-theory
of $C^*(Q)$, which we have now.
%
%
Theorem \ref{ktheorytheorem}, Lemma \ref{proofcondab} and the above discussion now yield the following theorem.

\begin{theorem} \label{TheoremKTheory}
Let $X$ be a weakly free $k$-semigraph algebra
(with $k < \infty$)
whose universal representation $X \longrightarrow C^*(X)$
is injective.
Suppose that the source projections cover the generators.
Then the semigraph $C^*$-algebra has the following $K$-theory:
$$K_1(C^*(X)) = 0, \qquad
K_0(C^*(X)) \cong  \ring(\calp)$$
via
$[p] \longleftrightarrow p$
for $p \in \calp$.
($\ring(\calp)$ denotes the subring of $X$ generated by $\calp$, regarded then as an abelian group under addition.)
\end{theorem}





We are going to write $\ring(\calp)$ as a direct limit of subrings by using common refinements
of projections in $\calp$.
For each finite subset $\calq$ of $\calp$ we consider the subring $\ring(\calq)$ generated by $\calq$.
This ring is generated by the base elements
\begin{equation}   \label{exprQA}
p_{\calq,A} = \prod_{p \in A} p \prod_{q \in \calq \backslash A} (1-q),
\end{equation}
where $A$ is a nonzero subset of $\calq$.
Note that these base elements are mutually orthogonal for different $A$'s.
The projection $q_{\calq,A}$
is nonzero if and only if
\begin{equation}   \label{exprQA2}
\Big (\prod_{p \in A} p \Big ) (1-q) \neq 0
\end{equation}
for all $q \in \calq \backslash A$.
This follows from Lemma \ref{lemmaEquivToeplitzCond}.
Let $\hat{\calq}$ denote the family of all subsets $A$ of $\calq$ for which $p_{\calq,A}$ is nonzero.
Since the the $p_{\calq,A}$'s are mutually orthogonal, we have
$\ring(\calq) = \Z^{|\hat{\calq}|}$.
We write $\ring(\calp)$ as a direct limit
\begin{equation}  \label{formulaRingP}
\ring(\calp) = \underrightarrow{\lim}_{\calq \subseteq \calp} \ring(\calq) = \underrightarrow{\lim}_{\calq \subseteq \calp} \Z^{|\hat{\calq}|}.
\end{equation}

\section{$K$-theory of freely generated semigraph $C^*$-algebras}

\label{SectionKTheoryQuellSemigraph}

We aim now to apply Theorem \ref{TheoremKTheory} to the freely generated semigraph $C^*$-algebra.
To this end consider the freely generated semigraph algebra $X$ associated to a finitely aligned
$k$-semigraph $T$. We may go over to its image $X'$ in $\calq(T)$, and write again $X$ rather than $X'$ for simplicity. This semigraph algebra is weakly free
by Theorem \ref{TheoremSemigraphCStarToeplitz} and Lemma \ref{corollaryFreeCancelCstar}.
By (\ref{freesemigraphP}) it is clear that the source projections cover the generators.
We can thus apply Theorem \ref{TheoremKTheory} if $k$ is finite.
If $k$ is infinite then we write $\calq(T)$ as the direct limit
$\calq(T) \cong \underrightarrow{\lim}_{k_0} C^*(X^{(k_0)})$,
where $k_0$ runs over the finite subsets of $k$, and
$X^{(k_0)}$ denotes the sub-semigraph algebra of $X$
which is generated by all elements of $T$ which have degree zero at any coordinate outside of $k_0$
(same proof as Lemma \ref{lemmaXbacki}).
Again, in $X^{(k_0)}$
the source projections cover the generators
by (\ref{freesemigraphP}).
Since $X$ is weakly free, $X^{(k_0)}$ is also weakly free.
So we can apply Theorem \ref{TheoremKTheory} to each
$C^*(X^{(k_0)})$ and get
\begin{equation}   \label{K0calqt}
K_0(\calq(T)) = \underrightarrow{\lim}_{k_0} K_0( C^*(X^{(k_0)}) ) =  \underrightarrow{\lim}_{k_0} \ring(\calp^{(k_0)}) = \ring(\calp).
\end{equation}






\begin{lemma}
If $s,t \in T$ then in $\calq(T)$ we have
\begin{eqnarray}
&&Q_t = Q_s \quad \Longleftrightarrow \quad  s =t,\\
&&  Q_t < Q_s  \quad
\Longleftrightarrow \quad s< t,   \label{smallerprojection0} \\
&& Q_s Q_t < Q_t \quad \mbox{ if } \quad s \not \le t .   \label{finerQ} 
\label{smallerprojection}
\end{eqnarray}

\end{lemma}

\begin{proof}
If $Q_s = Q_t$ then
by Lemma \ref{lemmaPalpha}
$\varphi(Q_s) \delta_{\mu_s} = \delta_{\mu_s} = \varphi(Q_t) \delta_{\mu_s}$ and thus $t \le s$.
Similarly, $s \le t$. We conclude that $s=t$ by Lemma \ref{lemmaOrderrel}.

If $Q_t < Q_s$ then by Lemma \ref{lemmaPalpha}
$\varphi(Q_t) \delta_{\mu_t} = \delta_{\mu_t} = \varphi(Q_s) \delta_{\mu_t}$ and thus $s \le t$.
Since $s \neq t$, we have $s < t$.
For the reverse implication, if $s < t$ then $t = \alpha s$ for some $\alpha \in T$. So $Q_t \le Q_s$. Since $s \neq t$,
we have $Q_t < Q_s$.

Suppose $Q_t Q_s = Q_t$. Then by Lemma \ref{lemmaPalpha}
$\varphi(Q_t Q_s) \delta_{\mu_t} = \varphi(Q_t) \delta_{\mu_t} = \delta_{\mu_t}$.
Consequently, $\varphi(Q_s) \delta_{\mu_t} \neq 0$, which implies that
$s \le t$.
\end{proof}

\begin{theorem}  \label{theoremKTheoryQuell}
Let $T$ be a $k$-semigraph. Then the freely generated semigraph $C^*$-algebra $\calq(T)$ has the following
$K$-theory.
If $T$ is finite then $C^*(T)$ is a finite dimensional $C^*$-algebra and
\begin{equation}  \label{k0QuellTheoFiniteT}
K_0(\calq(T)) \cong  \Z^{|\hat{\calp}|},   \qquad
K_1(\calq(T)) = 0.
\end{equation}
If $T$ is countably infinite then
\begin{equation}     \label{k0QuellTheo}
K_0(\calq(T)) \cong  \bigoplus_\N \Z,  \qquad
K_1(\calq(T)) = 0.
\end{equation}
%
%
\end{theorem}

\begin{proof}
We have already obtained (\ref{K0calqt}).
We aim to analyse $\ring(\calp)$ further. By (\ref{formulaRingP}) we need to estimate the size of the set $\hat \calq$ for a finite subset $\calq$ of $\calp$.

Suppose $T$ is finite. Then the set of standard words is a finite set.
Since they span $\calq(T)$, $\calq(T)$ is finite dimensional.
(\ref{k0QuellTheoFiniteT}) follows from (\ref{formulaRingP}).

Now suppose that $T$ is countably infinite.
Suppose $t \in T$ with $d(t)=n \delta_l$ for an $n \in \N$ and some $l \in k$.
Then set
$\calq_m= \{Q_{t(0,i \delta_l)}|\, 1 \le i \le m\}$
for $1 \le m \le n$.
Note that $\calq_m \subseteq \calq_n$.
Then $p_{\calq_n,\calq_m}$ of (\ref{exprQA}) is nonzero.
Indeed, for every $m < j \le n$ (\ref{exprQA2}) is
\begin{equation}
\Big (\prod_{i=1}^m Q_{t(0,i \delta_l)} \Big) (1-Q_{t(0,j \delta_l)}) = Q_{t(0,m \delta_l)} (1-Q_{t(0,j \delta_l)}),
\end{equation}
which is nonzero by (\ref{finerQ}).
Thus every $\calq_m$ is an element of $\hat \calq_n$. Consequently, $|\hat \calq_n| \ge n$ (or even exactly $n$ as one may check).
So, if $k$ is finite then we will find a sequence of $t$'s with $d(t)_{l} \rightarrow \infty$ for some coordinate $l \in k$ and so
%
(\ref{formulaRingP}) shows (\ref{k0QuellTheo}).

So suppose finally the case that $k=\N$ and there is an infinite family
of $t_i$'s such that $d(t_i)=\delta_i$ for all $i \in \N$.
Set $\calq_n = \{Q_{t_1}, \ldots, Q_{t_n}\}$ for $n \in \N$.
Set $A_m = \{Q_{t_m}\}$ for $1 \le m \le n$. Then $A_m \in \hat \calq_n$ since $p_{\calq_n,A_m}$ is nonzero.
Indeed,
(\ref{exprQA2}) is
$Q_{t_m} (1- Q_{t_j})$,
which is nonzero by (\ref{finerQ}).
So we have $|\hat \calq_n| \ge n$.
Again, (\ref{formulaRingP}) shows (\ref{k0QuellTheo}).
\end{proof}

\section{Concluding remarks}  \label{SectionConcludingRemarks}

In this section we remark which known results are generalised by this paper
and include some examples.
This was suggested by one of the referees.
Moreover, we shall revisit the definition of freely generated semigraph algebras.

Weakly freeness (see Definition \ref{defKToeplitz}) and freeness (see Definition \ref{defToeplitz})
generalise the Toeplitz conditions of Toplitz Cuntz algebras \cite{cuntz}, Toeplitz Cuntz--Krieger algebras \cite{cuntzkrieger},
finitely aligned higher rank Toeplitz graph algebras \cite{raeburnsims}
and higher rank Exel--Laca algebras under condition (II) \cite{burgiKTheoryGaugeActions} to the setting of higher rank semigraph algebras;
once directly, in that 
these Toeplitz versions of generalised Cuntz--Krieger algebras satisfy the freeness
condition, and on the other hand in a 
new way by covering algebras like the freely generated
semigraph algebras associated to a $k$-semigraph (Definition \ref{ToeplitzgrahDef}), which have not appeared in the literature
so far.
The assertion that freeness implies cancellation (Proposition \ref{PropToeplitzAmenability}), and so the validity of
a Cuntz--Krieger uniqueness theorem, recovers the corresponding uniqueness theorems for Toeplitz
Cuntz--Krieger algebras \cite{cuntzkrieger} and higher rank Toeplitz graph algebras \cite{raeburnsims}.

The $K$-theory computation of Theorem \ref{TheoremKTheory} generalises the $K$-theory computation of the higher rank Toeplitz
graph algebras and higher rank Exel--Laca algebras under condition (II) in \cite{burgiKTheoryGaugeActions};
these results, on the other hand, generalise Cuntz' computation of the $K$-theory of the Toeplitz
Cuntz--Krieger algebras in \cite{cuntz3}.
The $K$-theory computation of Theorem \ref{TheoremKTheory} 
may also apply to possible Toeplitz versions of ultragraph algebras \cite{tomforde}
and $C^*$-algebras of labelled graphs \cite{batespask}.
The computation of the $K$-theory of the freely generated semigraph algebras in Theorem \ref{theoremKTheoryQuell}
is new 
as these algebras have not been considered before.

Nevertheless, freely generated semigraph algebras are quite natural generalised Cuntz--Krieger algebras
associated to higher rank graphs. Let us revisit the axiomatic relations of a freely generated semigraph algebra (see Definition \ref{ToeplitzgrahDef}) for a
given $k$-graph $T$ which, for simplicity, we assume receives and emits only finitely many ($1 \le n < \infty$) edges in each vertex.
The relations are motivated by the following rules. (a) The set of $*$-words should form an inverse semigroup of partial isometries.
(b) The range projections $t t^*$ of a fixed degree $d(t)$ should
be mutually orthogonal for $t \in T$. (c) The sum $\sum_{t \in T, d(t)=n} t t^*$ should form a kind of a one-sided unit for sufficiently
many words 
(explained below).

The relations Definition \ref{ToeplitzgrahDef}.(i)-(v) are natural as to imply rule (a). Definition \ref{ToeplitzgrahDef}.(vi)
is driven by rules (b) and (c). How should we define the product $x^* y$ for $x,y \in T$? By rule (c)
we multiply from both sides with a unit, and end up, after applying the orthogonality rule (b), with
$$x^* y = \sum_{e \in T, \,|e|=|x| \vee |y|-|x|} \,\,\sum_{f \in T,\, |f|=|x| \vee |y|-|y|} e e^*  x^* y f f^*
= \sum_{xe = yf} e Q_{yf} f^*.$$
In Cuntz--Krieger algebras and graph algebras one requires that the source projection $Q$ is absorbed by range projections
$P$
(or they are orthogonal),	
and so the middle term $Q_{yf}$ would vanish. For freely generated semigraph algebra
we impose no further relations between the range and source projections. 
These considerations recover Definition \ref{ToeplitzgrahDef}.(vi), and we are done.

%

In the diagram below we visualise the relationship between source projections, indicated as empty circles, of emitting edges and
the range projections, indicated as filled circle, of receiving edges.
We see how the connection between source and range projections alter from graph algebras over Toeplitz algebras to freely
generated semigraph algebras. The two narrow empty circles in the ``Free" diagram should indicate the source projections
of two edges with a common source; the third circle corresponds to another source. The range projections are no longer
completely within the source projections, and the source projections become abelianly free among each other.
Nevertheless, the basic connection $Q_s P_t = 0 \Leftrightarrow st \in T$ holds in all three types of graph algebras
as a link between range and source projections.

%
%
%
%

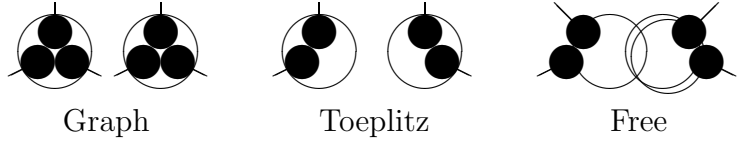
\vspace{.2cm}
\begin{center}

\begin{picture}(0, 0)(141,0)

\put(20,0){\circle{26}}

\put(20,7.5){\circle*{13}}
\put(26.495,-3.75){\circle*{13}}
\put(13.504,-3.75){\circle*{13}}

\put(20,7.5){\line(0,1){11}}
\put(26.495,-3.75){\line(2,-1){11}}
\put(13.504,-3.75){\line(-2,-1){11}}

\put(60,0){\circle{26}}

\put(60,7.5){\circle*{13}}
\put(66.495,-3.75){\circle*{13}}
\put(53.504,-3.75){\circle*{13}}

\put(60,7.5){\line(0,1){11}}
\put(66.495,-3.75){\line(2,-1){11}}
\put(53.504,-3.75){\line(-2,-1){11}}

\put(23,-30){Graph}	

\put(120,0){\circle{26}}

\put(120,7.5){\circle*{13}}
\put(113.504,-3.75){\circle*{13}}

\put(120,7.5){\line(0,1){11}}
\put(113.504,-3.75){\line(-2,-1){11}}

\put(160,0){\circle{26}}

\put(160,7.5){\circle*{13}}
\put(166.495,-3.75){\circle*{13}}

\put(160,7.5){\line(0,1){11}}
\put(166.495,-3.75){\line(2,-1){11}}

\put(120,-30){Toeplitz}	

\put(230,0){\circle{26}}

\put(220,7.5){\circle*{13}}
\put(213.504,-3.75){\circle*{13}}

\put(220,7.5){\line(-1,1){11}}
\put(213.504,-3.75){\line(-2,-1){11}}

\put(250,0){\circle{26}}
\put(252,-2){\circle{26}}

\put(260,7.5){\circle*{13}}
\put(266.495,-3.75){\circle*{13}}

\put(260,7.5){\line(1,1){11}}
\put(266.495,-3.75){\line(2,-1){11}}



\put(230,-30){Free}		

\end{picture}
\end{center}

\vspace{1.0cm}

The $K$-theory is not a good instrument to distinguish freely generated semigraph algebras. However, we have
a concrete model by partial isometries acting on a 
concrete Hilbert space, namely $C^*_r(\Delta)$.
So we 
focus on some examples of $k$-semigraphs.

\begin{itemize}
\item[(i)]
Typical examples of higher rank semigraphs are higher rank graphs $\Lambda$, the cut-down graph
$\Lambda^{(\le n)} = \{\lambda \in \Lambda|\, d(\lambda)\le n\}$ 
and $\Lambda/\Lambda^{(0)} \sqcup \{1\}$.

\item[(ii)]
Moreover, we may take any subset $\Lambda'$ of a $k$-graph $\Lambda$ and close it to a $k$-semigraph
$\overline{\Lambda'} := \{\lambda(m,n) \in \Lambda|\, \lambda \in \Lambda',\, 0 \le m\le n\le d(\lambda)\}$.

\item[(iii)]
For a shift of finite type with finite alphabet $\cala$ and finite forbidden word set $W \subset \cala^*$
the set $(\cala^*)_W$ of finite allowed words is a $1$-semigraph under concatenation and degree map being the length function
for words, which usually is not a $1$-graph.
Similarly as for the Cuntz--Krieger algebras (where every word in $W$ has length two), we may consider the one-sided shift $(\cala^\N)_W$ and consider
the homomorphism $\phi: \calq((\cala^*)_W) \rightarrow B(\ell^2((\cala^\N)_W))$ induced by shifting.
Note that for a sufficiently large $N$, every projection $\phi(Q_\lambda)$ with $d(\lambda)=N$ is the sum of some projections $\phi(P_\mu)$ with $d(\mu)=N$.
Hence, in $\calq((\cala^*)_W)$, we divide out the ideal $I$ generated by pulling back the relations $\phi(Q_\lambda P_\mu)= 0$ (if true) and $\phi(Q_\lambda P_\mu)=\phi(P_\mu)$ (if true)
for all $\lambda,\mu \in (\cala^*)_W$,
and the single relation 
$\sum_{\alpha \in \cala} P_\alpha = 1$ (this relation is also related to the full semigraph algebra
discussed in \cite{burgiSemigraphI}).
The quotient $\calq((\cala^*)_W)/I$ might be a good model of a Cuntz--Krieger algebra associated to the shift of finite type $(\cala,W)$.
Since we can shift-equivalently change any shift of finite type to another one with forbidden word length two, and so associate
a classical Cuntz--Krieger algebra, this example may not be completely helpful. 
%
But it demonstrates higher flexibility 
by considering semigraphs instead of graphs,
and 
by allowing more freedom in defining relations between source and range projections of generators.
\end{itemize}

%
%
%
%


\bibliographystyle{plain}
\bibliography{references}

\end{document}